\documentclass[12pt,reqno]{amsart}
\usepackage{amssymb}
\usepackage[usenames,dvipsnames]{color}
\usepackage{euscript}
\usepackage{multicol}
\usepackage{amssymb}
\usepackage{amsmath}
\usepackage{times} 
\usepackage{amssymb}
\usepackage{amsmath}
\usepackage{latexsym}
\usepackage{amsthm}
\usepackage{epsfig}
\usepackage{color}
\usepackage{comment}
\usepackage{amsfonts,amssymb,mathrsfs,amscd}
\usepackage{bm}

\usepackage{enumitem}

\newcommand\Tr{{\rm Tr\,}}

\newcommand{\be}[1]{\begin{equation}\label{#1}}
\newcommand{\ee}{\end{equation}}
\newtheorem{theorem}{Theorem}[section]
\newtheorem{corollary}{Corollary}

\newtheorem{remark}{Remark}

\numberwithin{equation}{section}

\title
[Magnetic Lieb--Thirring inequalities on the torus]
 {Magnetic Lieb--Thirring inequalities on the torus }

\author[A.Ilyin, A.Laptev] {Alexei  Ilyin,
 Ari Laptev}


\address
{\noindent\newline  Keldysh Institute of Applied Mathematics;
\newline
Imperial College London}
\email{ilyin@keldysh.ru; a.laptev@imperial.ac.uk}

\begin{document}

\maketitle

\medskip

\bigskip
\begin{quote}
{\normalfont\fontsize{8}{10}
\selectfont{\bfseries  Abstract.}
In this paper we prove Lieb--Thirring inequalities
for magnetic Schr\"odinger operators on the torus, where the constants in the inequalities depend on the magnetic flux.
}
\end{quote}

\setcounter{equation}{0}
\section{Introduction}\label{Sec:1}

Lieb--Thirring inequalities have important
applications in mathematical physics, analysis, dynamical systems,
attractors, to mention a few. A current state of the art  of many aspects
of the theory is presented in~\cite{lthbook}.

In certain applications Lieb--Thirring inequalities are  considered
on a compact manifold (e.\,g.,  torus, sphere~\cite{ILZ-JFA}). In this case one has to
impose the zero mean orthogonality condition. However, in the case of a torus
the corresponding constants in the Lieb--Thirring inequalities depend
on the aspect ratios of the periods, for example, on the 2D torus the
rate of growth of the constants  is proportional to the aspect ratio.

On the other hand, on the torus $\mathbb{T}^d$ with arbitrary periods
it is possible to obtain bounds for the Lieb--Thirring constants that
are independent of the ratios of the periods, provided that we impose
a stronger orthogonality condition that the functions must have zero average
over the shortest period uniformly with respect to the remaining variables~\cite{I-L-MS}.

In this work we prove Lieb--Thirring inequalities on the torus  for the magnetic
Laplacian.
The introduction of the magnetic potential not only  removes the orthogonality condition but makes it possible
to obtain bounds for the constants  that are independent of the periods of the torus
(more precisely, depend only on the corresponding magnetic fluxes).

In this paper, when obtaining the constant in the Lieb--Thirring inequality
 we use a combination of the result obtained in \cite{IMRN} and also adopting
 the proof from \cite{lthbook} to the case of the magnetic operator
 on the torus. Surprisingly both such independent estimates
 play important  and non-interchangeable roles depending on the magnetic fluxes.

In conclusion of this brief introduction we point out that magnetic interpolation
inequalities both in $\mathbb{R}^d$, and in the periodic case received much attention over the last
years, see \cite{DELL1,DELL2,Nazarov} and the references therein.

We now describe our main result.
Let $\mathbb{T}^d=\mathbb{T}^d(L)$  be the $d$-dimensional torus with periods
$L_1,\dots,L_d$.
Let us consider the eigenvalue problem for the magnetic
 Schr\"odinger
operator $\mathcal{H}$ in $L_2(\mathbb T^{d})$:
\begin{equation}\label{onTd}
\aligned
&\mathcal{H}\Psi=
\left(i\,\nabla_x - \mathbf A(x)\right)^2 \Psi-V(x)\Psi\\=
&\sum_{j=1}^d(i\partial_{x_j}-a_j(x_j))^2\Psi-V(x)\Psi=-\lambda\Psi,
\endaligned
\end{equation}
where
$$\mathbf A(x) = \left(a_1(x_1),\dots a_{d} (x_{d})\right)$$
is the real-valued  magnetic vector potential in the ``diagonal'' case when $a_j(x)=a_j(x_j)$.
For each $j$ we define the magnetic flux
$$
\alpha_j = \frac{1}{2\pi}\, \int_0^{L_j} a_j(x_j) \, dx_j  \qquad 1\le j \le d,
$$
and assume that $\alpha_j\not\in\mathbb Z$ for all $j$. Then we have the following result.

\begin{theorem}\label{Th:Tdspec} Suppose that the potential
$V(x)\ge0$  and $V\in L_{\gamma + d/2}(\mathbb T^{d})$.
Let $\gamma \ge1$. Then the following bound holds
for the $\gamma$-moments of the negative eigenvalues
of operator~\eqref{onTd}:
\begin{equation}\label{gamma-moments}
\sum_{n}\lambda_n^{\gamma}
\le L_{\gamma, d} \, \int_{\mathbb T^{d}}
 V^{\gamma + d/2} (x)\,dx ,
\end{equation}
where
\begin{equation}\label{gamma-moments-const}
L_{\gamma, d} \le \left(\frac{\pi}{\sqrt3}\right)^{d}
L_{\gamma,d}^{\mathrm{cl}} \prod_{j=1}^{d} \sqrt{\mathrm K(\alpha_j)}.
\end{equation}
Here $L_{\gamma,d}^{\mathrm{cl}}$ is the semiclassical constant~\eqref{class},
and
 $$
\mathrm{K}(\alpha)\le\min(\mathrm{K}_1(\alpha),\mathrm{K}_2(\alpha)).
$$
The expressions for $\mathrm{K}_1(\alpha)$ and $\mathrm{K}_2(\alpha)$ are as follows:
\begin{align}\label{K1}
&\mathrm{K}_1(\alpha)=\mathrm k(\alpha)^2,\quad \mathrm k(\alpha)=\left\{
            \begin{array}{ll}
             \frac1{|\sin(2\pi\alpha)|}, & \hbox{$0<\alpha\,\mathrm{mod}(1)<1/4$;} \\
              1, & \hbox{$1/4\le\alpha\,\mathrm{mod}(1)\le3/4$;} \\
             \frac1{|\sin(2\pi\alpha)|}, & \hbox{$3/4<\alpha\,\mathrm{mod}(1)<1$,}
            \end{array}
          \right.
\\
\label{K2}
&\mathrm K_2(\alpha)=\frac5{3\sqrt{3}\pi}
\cdot\left[\sup_{b\ge
0}b^{5/3}\sum_{k\in\mathbb{Z}} \frac1{(|k+\alpha|^3+b)^2}\right]^2\,.
\end{align}
\end{theorem}

In Sections~\ref{Sec:2} and \ref{Sec:3} we consider the one dimensional case, where
 this theorem is proved in the equivalent dual formulation
in terms of orthonormal systems in the scalar case and the matrix case, respectively.
We point out that this theorem with $\mathrm K(\alpha)\le\mathrm K_1(\alpha)$
was proved in \cite{IMRN} and the proof was based on the magnetic interpolation inequality~\eqref{magnetic} (whose proof is briefly recalled in Section~\ref{Sec:2}).
In the 1D scalar case  this inequality immediately gives the result by the method
of \cite{E-F}, while the inequality in the essential matrix case was
proved in \cite{IMRN} (see also~\cite{D-L-L} for the starting point of this approach).

The bound for the constant $\mathrm K(\alpha)\le\mathrm K_2(\alpha)$
was proved in the 1D scalar case in~\cite{UMN}. The proof
in the matrix case is given in Theorem~\ref{Th:matrix_orth}. Then the
inequalities for orthonormal systems are equivalently
reformulated in Theorem~\ref{Th:matrix_neg} in terms of estimates
for the negative trace  and for higher-order Riesz means of negative eigenvalues in
Corollary~\ref{1D gamma-moments}. Finally, Theorem~\ref{Th:Tdspec} is proved in
Section~\ref{Sec:4} by using the lifting argument with respect to dimensions~\cite{Lap-Weid}.
 The fact that the magnetic potential is of the special diagonal form
is crucial here.

We see in \eqref{K1} and \eqref{K2} that  unlike $\mathrm
K_1(\alpha)$, the constant $\mathrm K_2(\alpha)$ is not given in
the explicit form. A computation in Section~\ref{Sec:5} shows that
in the central region $|\alpha-1/2|<0.2273$ it holds that $\mathrm
K_2(\alpha)<\mathrm K_1(\alpha)$, while near the end-points
$\mathrm K_1(\alpha)$ is better, see Fig.~\ref{Fig.2}.

\setcounter{equation}{0}
\section{1D periodic case}\label{Sec:2}

We consider here the magnetic Lieb--Thirring inequality
in the 1D periodic case. We
assume that the period equals
$$
L=\frac{2\pi}\varepsilon,\quad\varepsilon>0.
$$
Of course, one can use scaling and consider only the case $\varepsilon=1$,
but we prefer to consider the general case in order to trace down
the corresponding constants in the most explicit way.
\begin{theorem}\label{Th:magnetic1d}
 Let the family of functions $\psi_1, \dots, \psi_N \in H^1([0,L]_\mathrm{per})$
 be orthonormal in $L^2([0,L]_\mathrm{per})$. Then
\begin{equation}\label{2}
\int_0^{L} \rho(x)^3\, dx \le {\mathrm K}(\alpha) \sum_{n=1}^N \int_{0}^{L} |i \psi'_n(x) - a(x) \psi_n(x)|^2 \, dx,
\end{equation}
where
$$
 \rho(x) = \sum_{n=1}^N |\psi_n(x)|^2
$$
 and
 $$
\mathrm{K}(\alpha)\le\min(\mathrm{K}_1(\alpha),\mathrm{K}_2(\alpha)).
$$
Here
$\alpha$ is the magnetic flux
\begin{equation}\label{flux}
\alpha:=\frac1{2\pi}\int_0^La(x)\,dx,
\end{equation}
and the constants $\mathrm{K}_1(\alpha)$ and $\mathrm{K}_2(\alpha)$ are defined in
\eqref{K1}, \eqref{K2}.
\end{theorem}
\begin{proof}
We first
point out that  estimate~\eqref{K1} was obtained in \cite[(6.8)]{IMRN},
where $\mathrm k(\alpha)$ is the constant in the
1D magnetic interpolation inequality
\begin{equation}\label{magnetic}
\|u\|_\infty^2\le \mathrm k(\alpha)\left(\int_0^{L}
\left|i\,u'(x)-a(x) u(x)\right|^2dx\right)^{1/2}
\left(\int_0^{L}
|u(x)|^2dx\right)^{1/2}.
\end{equation}
The sharp constant $\mathrm k(\alpha)$ (shown in Figure~\ref{Fig.2})
was found in \cite[(3.5)]{IMRN} and is given in~\eqref{K1}.
For the sake of completeness we briefly  recall the proof of \eqref{magnetic}.
We further assume for the moment that
the magnetic potential is constant $a(x)\equiv a=\mathrm{const}$.
We use the Fourier series
$$
\psi(x)=\sqrt{\frac\varepsilon{2\pi}}\sum_{k\in\mathbb{Z}}\widehat\psi_k e^{ik\varepsilon x},\qquad
\widehat{\psi}_k=\sqrt{\frac\varepsilon{2\pi}}
\int_{0}^{2\pi/\varepsilon}\psi(x)e^{-ik\varepsilon x}dx.
$$
We consider  the self-adjoint operator
$$
A(\lambda):=\left(i\frac d{dx}-a\right)^2+\lambda I
$$
and its Green's function $G_\lambda(x,\xi)$
$$
A(\lambda)G_\lambda(x,\xi)=\delta(x-\xi),
$$
which is found in terms of the Fourier series
\begin{equation}\label{Green}
G_\lambda(x,\xi)=\frac\varepsilon{2\pi}\sum_{n\in\mathbb Z}
\frac{e^{in\varepsilon(x-\xi)}}{(n\varepsilon+a)^2+\lambda}.
\end{equation}
On the diagonal we obtain
$$
\aligned
G(\lambda):=G_\lambda(\xi,\xi)=\frac\varepsilon{2\pi}
\sum_{n\in\mathbb Z}
\frac1{(n\varepsilon+a)^2+\lambda}\\=
\frac1\varepsilon\frac1{2\pi}\sum_{n\in\mathbb Z}
\frac1{(n+\alpha)^2+\lambda/\varepsilon^2}\\=
\frac1{2\sqrt{\lambda}}\,\frac{\sinh(2\pi\sqrt{\lambda}/\varepsilon)}
{\cosh(2\pi\sqrt{\lambda}/\varepsilon)-\cos(2\pi\alpha)}.
\endaligned
$$
Using a general result (see Theorem~2.2 in \cite{IZ} with $\theta=1/2$)
we find that the sharp constant in~\eqref{magnetic} is as follows
$$
\mathrm{k}(\alpha)=\frac1{\theta^\theta (1-\theta)^{1-\theta}}
\sup_{\lambda>0}\lambda^\theta G(\lambda)=\sup_{\varphi>0}F(\varphi),
$$
where
$$
F(\varphi)=\frac{\sinh(\varphi)}
{\cosh(\varphi)-\cos(2\pi \alpha)}\,,\quad\varphi=2\pi\sqrt{\lambda}/\varepsilon.
$$
An elementary analysis of the dependence of the behaviour of the function $F(\varphi)$
on the parameter $\alpha=a/\varepsilon$ (see \cite{IMRN} for the details) gives
the expression for $\mathrm k(\alpha)$ in \eqref{K1}.

We now consider the case of a non-constant magnetic potential
$a(x)$. It this case instead of the complex exponentials we
consider the orthonormal system of functions
\begin{equation}\label{nonconst}
\sqrt{\frac\varepsilon{2\pi}}\varphi_n(x),\quad \varphi_n(x)=
e^{i((n+\alpha)\varepsilon x-\int_0^xa(y)dy)}
\end{equation}
that are periodic with period $2\pi/\varepsilon$ in view of~\eqref{flux} and satisfy
$$
\left(i\frac d{dx}-a(x)\right)\varphi_n(x)=-\varepsilon(n+\alpha)\varphi_n(x).
$$
Therefore the Green's function of the operator $A(\lambda)$ is
$$
G_\lambda(x,\xi)=\frac\varepsilon{2\pi}\sum_{n\in\mathbb Z}
\frac{e^{i(n+\alpha)\varepsilon(x-\xi)-\int_\xi^xa(y)dy}}{\varepsilon^2(n+\alpha)^2+\lambda},
$$
giving the same expression for $G_\lambda(\xi,\xi)$ as in \eqref{Green}
and hence the same expression for  $\mathrm k(\alpha)$
as in the case $a(x)=\mathrm{const}$.

We can now obtain inequality~\eqref{2} with
$\mathrm{K}(\alpha)\le\mathrm k(\alpha)^2$ by the method
of~\cite{E-F}. For an arbitrary $\xi=(\xi_1,\dots,\xi_N)\in\mathbb{C}^N$
we set $u(x)=\sum_{n=1}^N\xi_n\psi_n(x)$ in \eqref{magnetic}. Using
orthonormality we obtain
$$
\biggl|\sum_{n=1}^N\xi_n\psi_n(x)\biggr|^4\le\mathrm k(\alpha)^2|\xi|^2\,
\sum_{n,k=1}^N\xi_n\bar\xi_{k}\bigl(i\psi'_n-a\psi_n,i\psi'_{k}-a\psi_{k}\bigr)
$$
For a fixed $x$ we set $\xi_j:=\bar\psi_j(x)$, $j=1,\dots,N$,  which gives
$$
\rho(x)^3\le\mathrm k(\alpha)^2\sum_{n,k=1}^N\psi(x)_n\bar\psi_{k}(x)\bigl(i\psi'_n-a\psi_n,i\psi'_{k}-a\psi_{k}\bigr).
$$
Integrating in $x$ and again using orthonormality we obtain~\eqref{2} with~\eqref{K1}.
\medskip

It now remains to prove \eqref{K2}: $\mathrm{K}(\alpha)\le\mathrm K_2(\alpha)$.
Let  $f$ be a non-negative function on $\mathbb{R}^+$ with $\int_0^\infty f(t)^2dt=1$ so that
$$
\int_0^\infty f(t/E)^2dt=E.
$$
Let $a(x)\ne\mathrm{const}$. We use the Fourier series with respect to
system~\eqref{nonconst}:
$$
\psi(x)=\sqrt{\frac\varepsilon{2\pi}}\sum_{k\in\mathbb{Z}}\widehat\psi_k \varphi_k(x),\qquad
\widehat{\psi}_k=\sqrt{\frac\varepsilon{2\pi}}
\int_{0}^{2\pi/\varepsilon}\psi(x)\varphi_k(-x)dx.
$$
Then we obtain that
$$
\aligned
\int_{0}^{L} |i \psi'(x) - a(x)\,  \psi(x)|^2 \, dx=
\sum_{k\in\mathbb{Z}}\varepsilon^2|k+\alpha|^2|\widehat\psi(k)|^2=\\
\int_0^\infty \sum_{k\in\mathbb{Z}} f\biggl(\frac E{\varepsilon^2|k+\alpha|^2}\biggr)^2|\widehat\psi_k|^2dE=
\int_{0}^{L}\int_0^\infty|\psi^E(x)|^2dEdx,
\endaligned
$$
where
$$
\psi^E(x)=\sqrt{\frac\varepsilon{2\pi}}\sum_{k\in\mathbb{Z}}
 f\biggl(\frac E{\varepsilon^2|k+\alpha|^2}\biggr)\widehat\psi_k\varphi_k(x).
$$
Therefore
\begin{equation}\label{hi}
\aligned
&\psi(x)-\psi^E(x)\\=&
\sqrt{\frac\varepsilon{2\pi}}\sum_{k\in\mathbb{Z}}
\left(1-f\biggl(\frac E{\varepsilon^2|k+\alpha|^2}\biggr)\right)\widehat\psi_k\varphi_k(x)=
(\psi(\cdot), \chi^E(\cdot,x)),
\endaligned
\end{equation}
where
\begin{equation}
\label{chi}
\chi^E(x',x)=\sqrt{\frac\varepsilon{2\pi}}
\sum_{k\in\mathbb{Z}}
\left(1-f\biggl(\frac E{\varepsilon^2|k+\alpha|^2}\biggr)\right)\varphi_k(x')\varphi_k(-x).
\end{equation}
For any $\delta>0$ we have
\begin{equation}\label{eps}
\aligned
\rho(x)\le
(1+\delta) \sum_{n=1}^N |\psi_n^{E}(x)|^2
+ (1+\delta^{-1}) \sum_{n=1}^N |\psi_n(x) - \psi_n^{E}(x)|^2.
\endaligned
\end{equation}
In view of orthonormality, Bessel's inequality, \eqref{hi}
and the fact that $|\varphi_k(x)|\equiv1$ we have
\begin{equation}\label{sumchi}
\aligned
\sum_{n=1}^N |\psi_n(x) - \psi_n^{E}(x)|^2=
\sum_{n=1}^N |(\psi_n(\cdot),\chi^E(\cdot,x))|^2\le\\\le
\|\chi^E(\cdot,x))\|^2_{L^2(0,L)}=
\frac\varepsilon{2\pi}\sum_{k\in\mathbb{Z}}
\left(1-f\biggl(\frac E{\varepsilon^2|k+\alpha|^2}\biggr)\right)^2.
\endaligned
\end{equation}
Next, following \cite{Frank-Nam, lthbook} (see Remark~\ref{R:optch}) we set
\begin{equation}\label{choff}
f(t)=\frac1{1+\mu  t^{3/2}},\qquad\mu=\left(\frac{4\pi}{9\sqrt{3}}\right)^{3/2}\,.
\end{equation}
This gives
\begin{equation}\label{A(alpha)}
\aligned
\|\chi^E(\cdot,x))\|^2=\frac1{2\pi}\varepsilon\mu^2E^3\sum_{k\in\mathbb{Z}}
\frac1{(\varepsilon^3|k+\alpha|^3+\mu E^{3/2})^2}=\\=
\frac1{2\pi}\varepsilon^{-5}\mu^2\sqrt{E}E^{5/2}\sum_{k\in\mathbb{Z}}
\frac1{(|k+\alpha|^3+\mu (\sqrt{E}/\varepsilon)^3)^2}=\\=
\frac1{2\pi}\mu^{1/3}\sqrt{E}\cdot b^{5/3}\sum_{k\in\mathbb{Z}}
\frac1{(|k+\alpha|^3+b)^2}\le\\\le
\frac1{2\pi}\mu^{1/3}\sqrt{E}\cdot\sup_{b\ge0}b^{5/3}\sum_{k\in\mathbb{Z}}
\frac1{(|k+\alpha|^3+b)^2}=\\=
\frac1{3^{5/4}\pi^{1/2}}\sqrt{E}\cdot\sup_{b\ge 0}b^{5/3}\sum_{k\in\mathbb{Z}}
\frac1{(|k+\alpha|^3+b)^2}=: A(\alpha)\sqrt{E},
\endaligned
\end{equation}
where
$$
 A(\alpha)=\frac1{3^{5/4}\pi^{1/2}}\cdot
 \sup_{b\ge 0}b^{5/3}\sum_{k\in\mathbb{Z}} \frac1{(|k+\alpha|^3+b)^2},
$$
and where we  singled out the factor $\sqrt{E}$,  set $b:=\mu E^{3/2}/\varepsilon^3$,
and recalled the definition of $\mu$.

Substituting this into \eqref{eps} and optimizing with respect to
$\delta$ we obtain
$$
\rho(x)\le\left(\sqrt{\sum_{n=1}^N|\psi^E_n(x)|^2}+\sqrt{A(\alpha)}E^{1/4}\right)^2,
$$
which gives that
$$
\sum_{j=1}^N|\psi^E_j(x)|^2\ge\left(\sqrt{\rho(x)}-\sqrt{A(\alpha)}E^{1/4}\right)^2_+.
$$
Finally,
$$
\aligned
\int_{0}^{L} |i \psi'(x) - a(x)\psi(x)|^2 \, dx=
\int_{0}^{L}\int_0^\infty|\psi^E(x)|^2dEdx\ge\\\ge
\int_{0}^{L}\int_0^\infty\left(\sqrt{\rho(x)}-\sqrt{A(\alpha)}E^{1/4}\right)^2_+dEdx=
\frac1{15A(\alpha)^2}\int_0^{L}\rho(x)^3dx.
\endaligned
$$
The proof is complete.
\end{proof}

\begin{remark}\label{R:optch}
{\rm
The series over $k\in\mathbb{Z}$ in~\eqref{sumchi}, which we obviously want
to minimize under the condition $\int_0^\infty f(t)^2dt=1$,  corresponds
(after the change of variable $t\to\ t^{-1/2}$) to the
integral
$$
\int_0^\infty(1-f(t))^2t^{-3/2}dt.
$$
A more general problem
$$
\int_0^\infty(1-f(t))^2t^{-\beta}dt\to\inf
$$
subject to the same condition $\int_0^\infty f(t)^2dt=1$
 was solved in
\cite{Frank-Nam}:
\begin{equation}\label{choice-of-f}
f(t)=\frac1{1+\mu t^\beta}, \qquad \mu=\left(\frac{\beta-1}\beta\cdot
\frac{\pi/\beta}{\sin(\pi/\beta)}\right)^\beta,\ \ \beta>1.
\end{equation}
This explains the choice of $f(t)$ in \eqref{choff}.
}
\end{remark}

\setcounter{equation}{0}
\section{1D periodic case for matrices}\label{Sec:3}

Let $\{\bm{\psi}_n\}_{n=1}^N$ be an orthonormal family of
vector-functions
$$
\bm{\psi}_n(x)=(\psi_n(x,1),\dots,\psi_n(x,M))^T, \quad \bm\psi_n:[0,L]_{\mathrm{per}}\to\mathbb{C}^M
$$
and
$$
\aligned
&(\bm\psi_n,\bm\psi_m):=
(\bm\psi_n,\bm\psi_m)_{L^2([0,L],\mathbb{C}^M)}\\=&
\sum_{j=1}^M\int_0^L\psi_n(x,j)\overline{\psi_m(x,j)}dx
=\int_0^L \bm\psi_n(x)^T\overline{\bm\psi_m(x)}dx=\delta_{nm}.
\endaligned
$$

We consider the $M\times M$ matrix $U(x)$
$$
U(x)=\sum_{n=1}^N\bm\psi_n(x)\overline{\bm\psi_n(x)}^T.
$$

\begin{theorem}\label{Th:matrix_orth}
The following inequality holds
\begin{equation}\label{orth_mag}
\int_0^{L}\Tr[U(x)^{3}]dx\le\mathrm  K(\alpha)\sum_{n=1}^N
\int_0^{L}|i \bm\psi'_n(x) -a(x) \bm\psi_n(x)|^2_{\mathbb{C}^M} dx,
\end{equation}
where $\mathrm K(\alpha)$ is defined in Theorem~\ref{Th:Tdspec}.
\end{theorem}
\begin{proof}
We first show that  $\mathrm K(\alpha)\le\mathrm{K}_2(\alpha)$.

As before, let $f$ be a scalar function with
$\int_0^\infty (t)^2dt=1$. Then

$$
\aligned
\int_{0}^{L} |i \bm\psi'(x) - a(x)\bm\psi(x)|_{\mathbb{C}^M}^2 \, dx=
\sum_{k\in\mathbb{Z}}\varepsilon^2|k+\alpha|^2|\widehat{\bm\psi}(k)|^2_{\mathbb{C}^M}=\\
\int_0^\infty \sum_{k\in\mathbb{Z}}
f\biggl(\frac E{\varepsilon^2|k+\alpha|^2}\biggr)^2|\widehat{\bm\psi}_k|^2_{\mathbb{C}^M}dE=
\int_{0}^{L}\int_0^\infty|\bm\psi^E(x)|_{\mathbb{C}^M}^2dEdx,
\endaligned
$$
where
$$
\bm\psi^E(x)=\sqrt{\frac\varepsilon{2\pi}}
\sum_{k\in\mathbb{Z}}f\left(\frac E{\varepsilon^2|k-\alpha|}\right)
\varphi_k(x)\widehat{\bm\psi}_k,\
\widehat{\bm\psi}_k=(\widehat{\psi}_k(1),\dots,\widehat{\psi}_k(M))^T.
$$

 Let
$\mathbf{e}\in\mathbb{C}^M$ be a constant vector. Then
$$
\aligned
\langle U(x)\mathbf{e},\mathbf{e}\rangle=
\sum_{n=1}^N|\mathbf{e}^T\bm\psi_n(x)|^2=
\sum_{n=1}^N|\langle\bm\psi_n(x),\mathbf{e}\rangle|^2\\=
\sum_{n=1}^N|\langle\bm\psi_n(x)-\bm\psi_n^E(x),\mathbf{e}\rangle+\langle\bm\psi_n^E(x),\mathbf{e}\rangle|^2\\\le
(1+\delta)\sum_{n=1}^N|\langle\bm\psi_n(x)-\bm\psi_n^E(x),\mathbf{e}\rangle|^2+
(1+\delta^{-1})\sum_{n=1}^N|\langle\bm\psi_n^E(x),\mathbf{e}\rangle|^2
\endaligned
$$
where $\langle\cdot,\cdot\rangle$
denotes the scalar product in $\mathbb{C}^M$. For the first term we have
$$
\aligned
\langle\bm\psi(x)-\bm\psi^E(x),\mathbf{e}\rangle\\=
\sqrt{\frac\varepsilon{2\pi}}\sum_{k\in\mathbb{Z}}
\left(1-f\left(\frac E{\varepsilon^2|k-\alpha|^2}\right)\right)
\varphi_k(x)\langle\widehat{\bm\psi}_k,\mathbf e\rangle\\=
(\bm\psi(\cdot),\chi^E(\cdot,x)\mathbf e)_{L^2(L,\mathbb C^M)},
\endaligned
$$
where the scalar  function $\chi^E(x',x)$ is as in~\eqref{chi}. Now, again by orthonormality,
Bessel's inequality and \eqref{A(alpha)} we obtain
$$
\aligned
\sum_{n=1}^N|\langle\bm\psi_n(x)-\bm\psi_n^E(x),\mathbf{e}\rangle|^2=
\sum_{n=1}^N(\bm\psi_n(\cdot),\chi^E(\cdot,x)\mathbf e)_{L^2(L,\mathbb C^M)}\\\le
\|\chi^E(\cdot,x)\mathbf e\|^2_{L^2(L,\mathbb C^M)}=
\|\chi^E(\cdot,x)\|^2_{L^2}\|\mathbf e\|_{\mathbb{C}^M}^2\le
A(\alpha)\sqrt{E}\|\mathbf e\|_{\mathbb{C}^M}^2.
\endaligned
$$
For the second term we simply write
$$
\sum_{n=1}^N|\langle\bm\psi_n(x),\mathbf{e}\rangle|^2=\langle U^E(x)\mathbf e,\mathbf e\rangle,
\quad
U^E(x)=\sum_{n=1}^N\bm\psi^E_n(x)\overline{\bm\psi^E_n(x)}^T.
$$
Combining the above we obtain
$$
\aligned
\langle U(x)\mathbf{e},\mathbf{e}\rangle\le
(1+\delta^{-1})\langle U^E(x)\mathbf{e},\mathbf{e}\rangle+
(1+\delta)A(\alpha)\sqrt{E}\|\mathbf e\|_{\mathbb{C}^M}^2.
\endaligned
$$
If we denote by $\lambda_j(x)$ and $\lambda_j^E(x)$, $j=1,\dots,M$  the
eigenvalues of the (Hermitian) matrices $U(x)$ and $U^E(x)$,
respectively, then the variational principle implies that
$$
\lambda_j(x)\le(1+\delta^{-1})\lambda_j^E(x)+(1+\delta)A(\alpha)\sqrt{E}.
$$
Optimizing with respect to $\delta$ we find that
$$
\lambda_j(x)\le\left(\sqrt{\lambda_j^E(x)}+A(\alpha)^{1/2}E^{1/4}\right)^2,
$$
or
$$
\lambda_j^E(x)\ge\left(\sqrt{\lambda_j(x)}-A(\alpha)^{1/2}E^{1/4}\right)^2_+,
\quad j=1,\dots, M.
$$
Therefore
$$
\sum_{n=1}^N|\bm\psi_n^E(x)|^2_{\mathbb{C}^M}=\Tr_{\mathbb{C}^M}U^E(x)\ge
\sum_{j=1}^M\left(\sqrt{\lambda_j(x)}-A(\alpha)^{1/2}E^{1/4}\right)^2_+.
$$
Integration with respect to $E$ gives that
$$
\aligned
&\sum_{n=1}^N\int_0^\infty|\bm\psi_n^E(x)|^2_{\mathbb{C}^M}dE\ge
\sum_{j=1}^M\int_0^\infty
\left(\sqrt{\lambda_j(x)}-A(\alpha)^{1/2}E^{1/4}\right)^2_+dE\\=&
\frac1{15A(\alpha)^2}\sum_{j=1}^M\lambda_j(x)^3=\frac1{15A(\alpha)^2}
\Tr U(x)^3,
\endaligned
$$
and integration with respect to $x$ gives~\eqref{orth_mag}
with~\eqref{K2}.

We finally point out that matrix inequality~\eqref{orth_mag}
with estimate of the constant~\eqref{K1} was previously proved in
\cite[Theorem~6.2]{IMRN}. The proof given there holds formally for the case
of a constant magnetic potential. However, if $a(x)\ne\mathrm{const}$ we only have to use
the orthonormal family~\eqref{nonconst} as we have done
in the proof of the scalar  Lieb--Thirring inequality in
Theorem~\ref{Th:magnetic1d}. The proof is complete.
\end{proof}

It is well known \cite{D-L-L, lthbook} that inequalities for orthonormal systems
are equivalent to the estimates for the negative trace
of the corresponding  Schr\"odinger operator.
In our case we consider the magnetic Schr\"odinger operator
\begin{equation}\label{H}
H=\left(i\frac d{dx}-a(x)\right)^2-V
\end{equation}
in $L_2([0,L]_{\mathrm{per}})$ with matrix-valued potential $V$.

\begin{theorem}\label{Th:matrix_neg}
Let $V(x)\ge0$ be an $M\times M$ Hermitian matrix such that $\Tr V^{3/2}\in L^1(0,L)$.
Then the spectrum of operator \eqref{H} is discrete and
the negative eigenvalues $-\lambda_n\le0$ satisfy the estimate
\begin{equation}\label{neg_mag}
\aligned
\sum_n\lambda_n\le\frac2{3\sqrt{3}}\sqrt{\mathrm K(\alpha)}\int_0^L\Tr[V(x)^{3/2}]dx\\=
\frac\pi{\sqrt{3}}\sqrt{\mathrm K(\alpha)}L_{1,1}^{\mathrm{cl}}\int_0^L\Tr[V(x)^{3/2}]dx,
\endaligned
\end{equation}
where
\begin{equation}\label{class}
L_{\gamma,d}^{\mathrm{cl}} = \frac{1}{(2\pi)^d}\, \int_{\mathbb R^d}  (1-|\xi|^2)_+^\gamma\, d\xi=
\frac{\Gamma(\gamma+1)}{2^d\pi^{d/2}\Gamma(\gamma+d/2+1)}.
\end{equation}
\end{theorem}
\begin{proof}
Let $\{\bm\psi_n\}_{n=1}^N$ be the orthonormal vector valued eigenfunctions
corresponding to $\{-\lambda_n\}_{n=1}^N$:
$$
\left(i\frac d{dx}-a(x)\right)^2\bm\psi_n-V\bm\psi_n=-\lambda_n\bm\psi_n.
$$
Taking the scalar product with $\bm\psi_n$, using inequality  \eqref{orth_mag}, H\"older's inequality for traces and setting
$X=\int_0^L\Tr[U(x)^3]dx$, we obtain
$$
\aligned
&\sum_{n=1}^N\lambda_n=
-\sum_{n=1}^N   \int_0^{L} |(i \bm\psi'(x) -a(x)\bm\psi_n(x)|^2_{\mathbb{C}^M}dx
+\int_0^{L}\Tr [V(x)U(x)]dx\\
&\le-\mathrm K(\alpha)^{-1}X+
\left(\int_0^L\Tr[V(x)^{3/2}]dx\right)^{2/3}
X^{1/3}.
\endaligned
$$
Calculating the maximum with respect to $X$ we
obtain~\eqref{neg_mag}.
\end{proof}

 The higher-order Riesz means
of the eigenvalues for magnetic Schr\"odinger operators with matrix-valued potentials
are obtained by the Aizenmann--Lieb argument \cite{AizLieb,lthbook}.
\begin{corollary}\label{1D gamma-moments}
Let $V\ge0$ be a  $M\times M$ Hermitian matrix, such that
$\Tr V^{\gamma + 1/2} \in L_1(0,L)$.
Then for any $\gamma \ge 1$  the negative eigenvalues of
the operator \eqref{H} satisfy the inequalities
\begin{equation}\label{1Dgamma}
\sum \lambda_n^\gamma \le L_{\gamma,1} \int_0^{L} \Tr[V(x)^{1/2 + \gamma}] \, dx,
\end{equation}
where
\begin{equation*}
L_{\gamma,1} \le \frac 2{3\sqrt{3}}\sqrt{\mathrm K(\alpha)}\, \frac{L_{\gamma,1}^{\mathrm{cl}}}{L_{1,1}^{\mathrm{cl}}}=
\frac\pi{\sqrt{3}}\sqrt{\mathrm K(\alpha)}\,L_{\gamma,1}^{\mathrm{cl}}.
\end{equation*}
\end{corollary}

\setcounter{equation}{0}
\section{Magnetic Schr\"odinger operator on the torus}\label{Sec:4}

\begin{proof}[Proof of Theorem~\ref{Th:Tdspec}]
We use the lifting argument with respect to dimensions developed in~\cite{Lap-Weid}.
More precisely, we apply estimate \eqref{1Dgamma} $d-1$ times with respect to
variables $x_1,\dots,x_{d-1}$ (in the matrix case), so that $\gamma$ is increased by $1/2$ at
each step, and, finally, we use~\eqref{1Dgamma} (in the scalar case) with respect to $x_d$.
Using the variational principle and
denoting the negative parts of the operators by $[\,\cdot\,]_-$ we obtain
\begin{align*}
&\sum_n\lambda_n^\gamma(\mathcal{H})
=\sum_n\lambda_n^\gamma\biggl((i\partial_{x_1}-a_1(x_1))^2+
\sum_{j=2}^{d-1}(i\partial_{x_j}-a_j(x_j))^2
-V(x)\biggr)
\\
&\ \
\le\sum_n\lambda_n^\gamma\biggl((i\partial_{x_1}-a_1(x_1))^2
-\biggl[\sum_{j=2}^{d-1}(i\partial_{x_j}-a_j(x_j))^2-
 V(x)\biggr]_-\biggr)
\\
&\ \
\le\frac\pi{\sqrt{3}}\sqrt{\mathrm K_1(\alpha_1)}{L}^\mathrm{cl}_{\gamma,1}
\int_0^{L_1}\biggl[(i\partial_{x_j}-a_j(x_j))^2
-V(x)\biggr]_-^{\gamma+1/2}\!dx_1
\\
&\ \ \le\dotsb\dotsb\dotsb\ \le
\biggl(\frac\pi{\sqrt{3}}\biggr)^{d-1}\ \prod_{j=1}^{d-1}\sqrt{\mathrm K(\alpha_j)}
\prod_{j=1}^{d-1}{L}^\mathrm{cl}_{\gamma+(j-1)/2,1}\times
 \\
& \ \times
\int_0^{L_1}\dotsi\int_0^{L_{d-1}}\operatorname{Tr}
\bigl[(i\partial_{x_d}-a_d(x_d))^2-V(x)\bigr]_-^{\gamma+(d-1)/2}\,
dx_1\dots dx_{d-1}
\\
&\ \
\le\biggl(\frac\pi{\sqrt{3}}\biggr)^{d}\ \prod_{j=1}^{d}\sqrt{\mathrm K(\alpha_j)}
\prod_{j=1}^{d}{L}^\mathrm{cl}_{\gamma+(j-1)/2,1}
\int_{\mathbb{T}^d}V^{\gamma+d/2}(x)\,dx,
\end{align*}
which proves \eqref{gamma-moments}, \eqref{gamma-moments-const}, since
$$
\prod_{j=1}^{d}{L}^\mathrm{cl}_{\gamma+(j-1)/2,1}={L}^\mathrm{cl}_{\gamma,d}.
$$
\end{proof}
\begin{remark}\label{R:orth2}
{\rm
The method of Theorem~\ref{Th:magnetic1d} (namely, its  second part) is difficult to apply in
the case orthonormal system on the torus $\mathbb{T}^d$  with $d>1$, because the corresponding series~\eqref{A(alpha)} is
now over the lattice $\mathbb{Z}^d$ and depends on $d$ parameters. However, the Lieb--Thirring
inequality for an orthonormal system $\{\psi_j\}_{j=1}^N\in H^1(\mathbb{T}^d)$ follows from
Theorem~\ref{Th:Tdspec}$_{\gamma=1}$ by duality. For example, for $d=2$ it holds
$$
\int_{\mathbb{T}^2}\rho(x)^2dx\le
\frac\pi6\sqrt{\mathrm K(\alpha_1)\mathrm K(\alpha_2)}
\sum_{j=1}^N\int_{\mathbb{T}^2}|i\nabla \psi_j(x)-\mathbf A(x)\psi_j(x)|^2_{\mathbb{C}^2}dx.
$$
}
\end{remark}

\setcounter{equation}{0}
\section{Some computations}\label{Sec:5}

We now present some computational results. We denote by
$F(b,\alpha)$ the key function in \eqref{A(alpha)}:
\begin{equation}\label{F}
F(b,\alpha):=b^{5/3}\sum_{k\in\mathbb{Z}}
\frac1{(|k+\alpha|^3+b)^2}.
\end{equation}
We clearly have that for all $\alpha$ (including integers)
$$
\lim_{b\to\infty}F(b,\alpha)=2\int_{0}^\infty\frac{dx}{(x^3+1)^2}=
\frac8{27}\sqrt{3}\pi=1.6122.
$$
This immediately gives  in the framework of this approach  (see~\eqref{K2}) a lower bound for the constant
${\mathrm K}_2(\alpha)$:
\begin{equation}\label{LTK}
{\mathrm K}_2(\alpha)\ge
\frac5{3\sqrt{3}\pi}\cdot\left[\frac8{27}\sqrt{3}\pi\right]^2=\frac{320\pi}{3^{13/2}}=
0.7961.
\end{equation}
 The graphs of $F(b,\alpha)$ for $\alpha=0.1, 0.25$, and $0.5$ are shown
in Fig. \ref{Fig.1}.
\begin{figure}[htb]
\centerline{\psfig{file=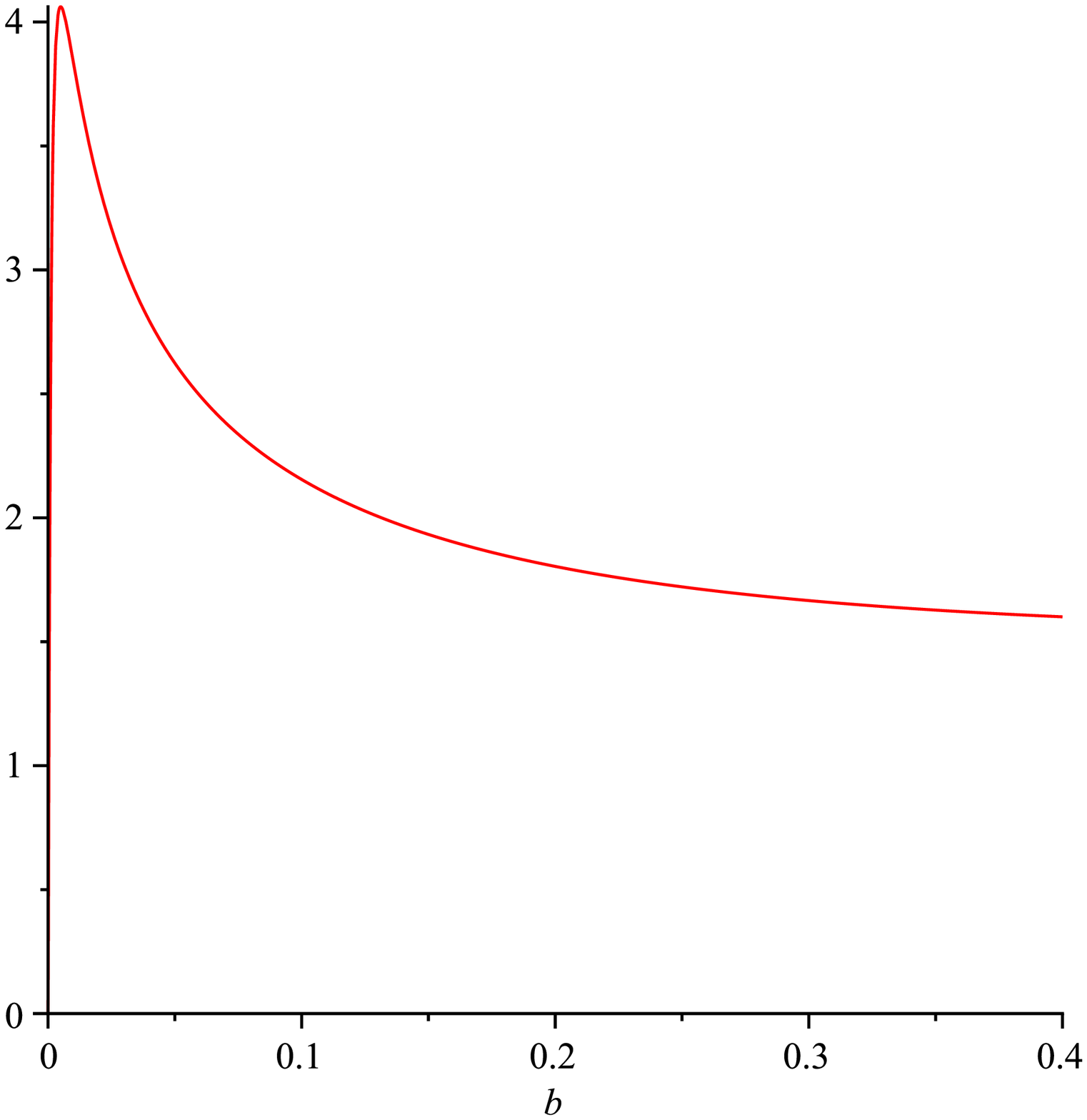,width=4.1cm,height=6cm,angle=0}
\psfig{file=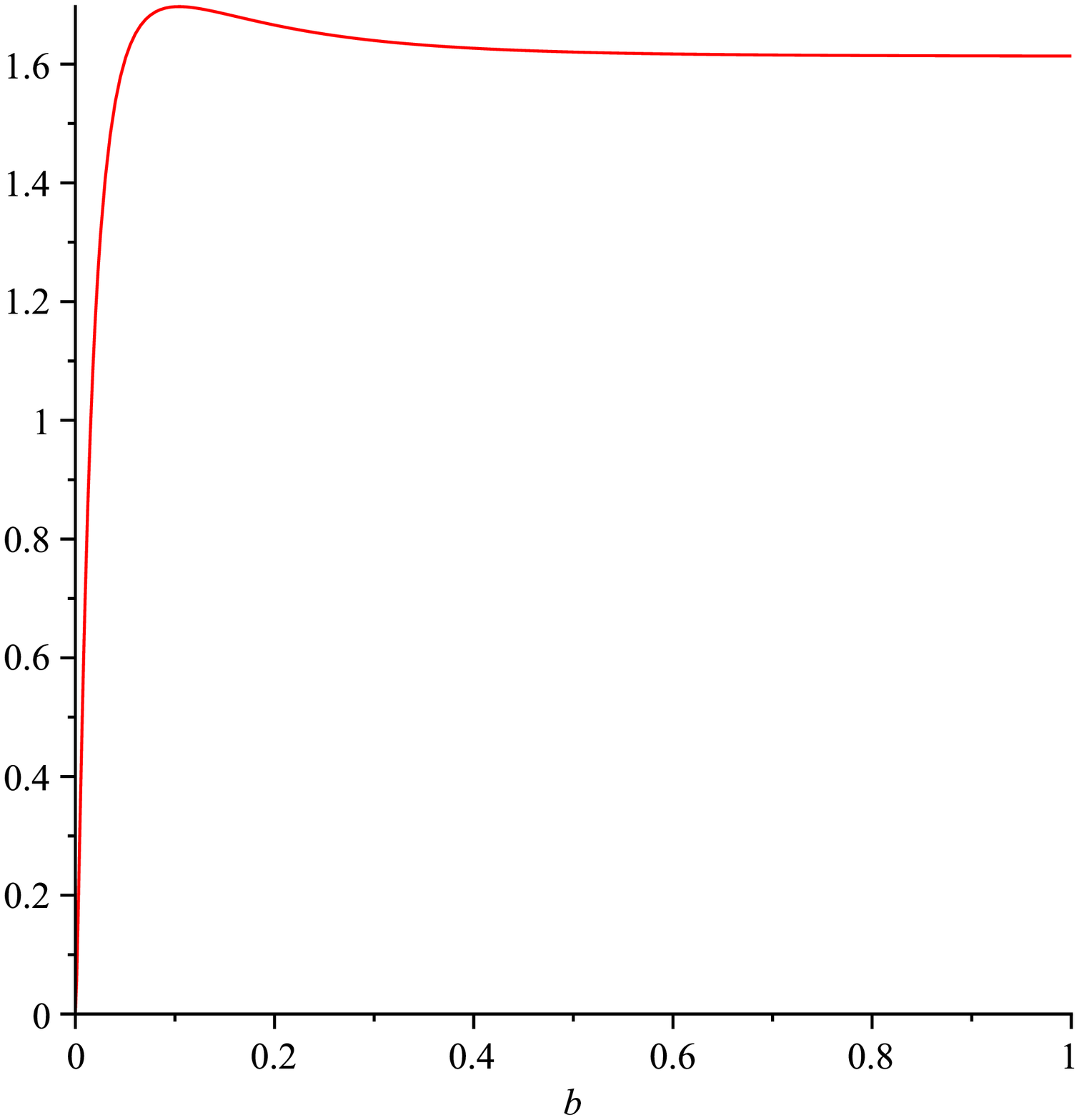,width=4.1cm,height=6cm,angle=0}
\psfig{file=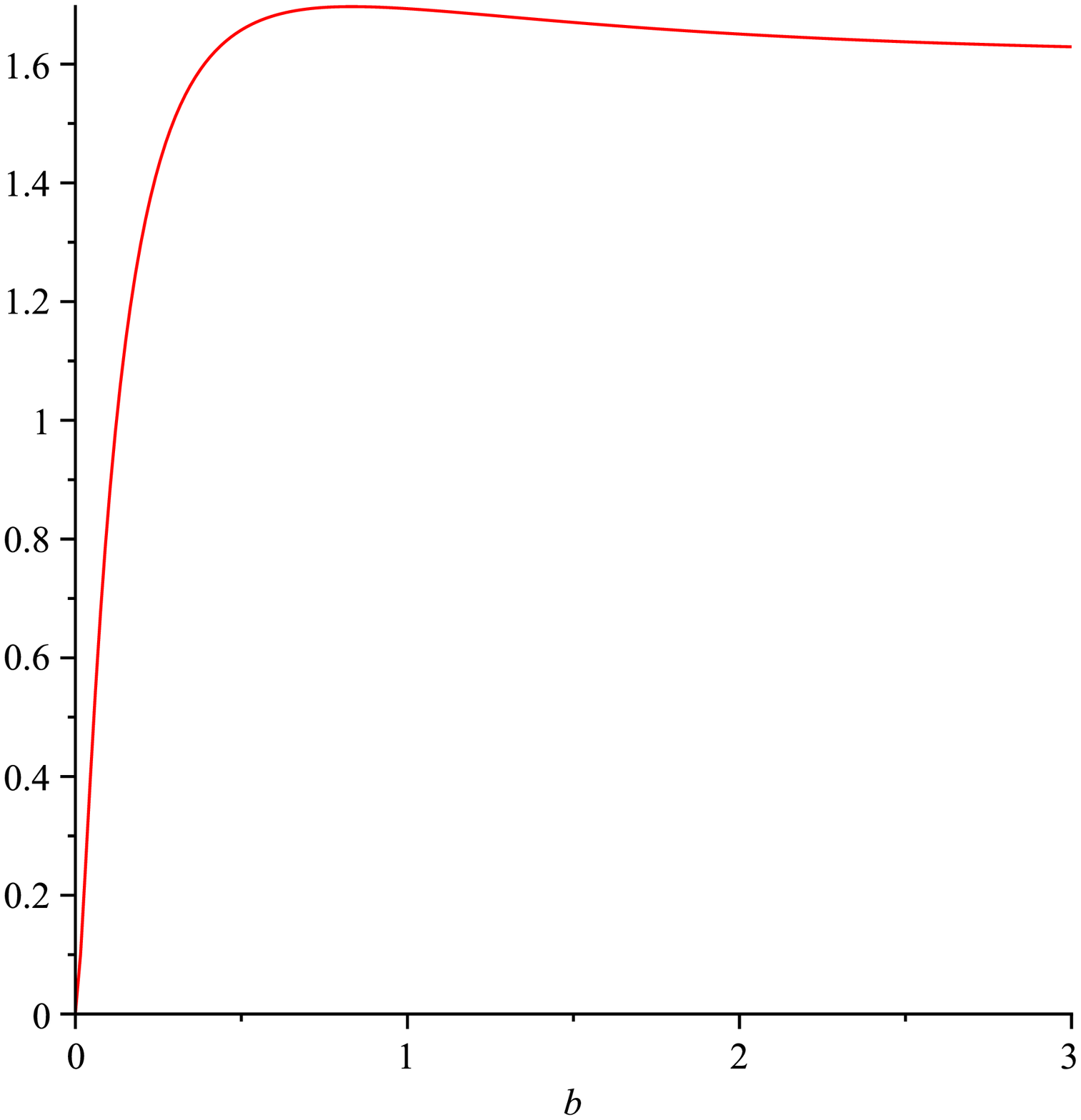,width=4.1cm,height=6cm,angle=0}}
\caption{The graphs of $F(b,0.1)$, $F(b,0.25)$, and $F(b,0.5)$.}
\label{Fig.1}
\end{figure}

The unique point of maximum $b^*(\alpha)$ has the following asymptotic
behaviour as $\alpha\to0^+$. For a small $\alpha$ the
main contribution in the sum in \eqref{F} comes from the  term with $k=0$,
that is, from
$$
b^{5/3} \frac1{(\alpha^3+b)^2},
$$
whose global maximum is attained at $b=5\alpha^3$ and equals
$\frac{5^{5/3}}{36}\cdot\frac1\alpha$.
Then \eqref{K2} gives
$$
\mathrm{K}_2(\alpha)\sim \frac{5^{13/3}}{3^{3/2}\,6^4\,\pi}\frac1{\alpha^2}=
0.0505\frac1{\alpha^2}\
\text{as}\ \alpha\to0,
$$
while it follows from  \eqref{K1} that
$$
\mathrm{K}_1(\alpha)\sim\frac1{4\pi^2}\frac1{\alpha^2}=0.025\frac1{\alpha^2}\
\text{as}\ \alpha\to0,
$$
which explains why $\mathrm K_1(\alpha)<\mathrm {K}_2(\alpha)$ near $\alpha=0$ and $\alpha=1$
in Fig. \ref{Fig.2}. On the other hand, in the middle region $|\frac12-\alpha|\le 0.2273$
the new estimate   \eqref{K2} is better.
It is also worth pointing out that
$$
\mathrm {K}_2(1/2)=\mathrm {K}_2(1/4)(=0.8819)
$$
the equality holding since
$$
F(b,1/2)=F(b/8,1/4).
$$

The minimum with respect to $\alpha$ is attained at $\alpha^*=0.273$ giving
$\mathrm {K}_2(\alpha^*)=\mathrm {K}_2(1-\alpha^*)=0.811$.
\begin{figure}[htb]
\centerline{\psfig{file=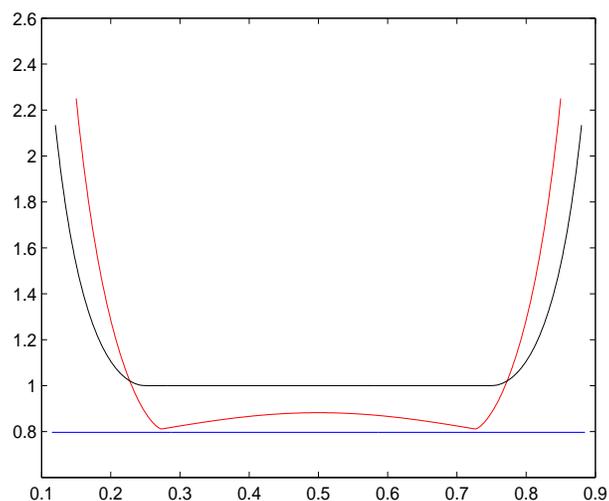,width=9.5cm,height=7.5cm,angle=0}}
\caption{The graphs of $\mathrm K_1(\alpha)=\mathrm k(\alpha)^2$ (black)
and $\mathrm{K}_2(\alpha)$ (red). The
horizontal blue line is the constant in~\eqref{LTK}.}
\label{Fig.2}
\end{figure}

\end{document}